\documentclass[11pt]{article}
\usepackage[T1]{fontenc}
\usepackage{lmodern}
\usepackage{microtype}
\usepackage[a4paper,margin=29mm]{geometry}
\usepackage{amsmath,amssymb,amsthm,mathtools}
\usepackage{xcolor}
\usepackage{hyperref}
\usepackage[nameinlink,capitalise,noabbrev]{cleveref}
\hypersetup{colorlinks=true,linkcolor=blue!45!black,citecolor=blue!45!black,
urlcolor=blue!55!black,pdftitle={Largest bulk gap of the complex Ginibre ensemble}}
\allowdisplaybreaks
\newtheorem{theorem}{Theorem}[section]
\newtheorem{proposition}[theorem]{Proposition}
\newtheorem{lemma}[theorem]{Lemma}
\newtheorem{corollary}[theorem]{Corollary}
\theoremstyle{remark}
\newtheorem{remark}[theorem]{Remark}
\newcommand{\C}{\mathbb C}
\newcommand{\R}{\mathbb R}
\newcommand{\Pp}{\mathbb P}
\newcommand{\dd}{\,\mathrm d}
\newcommand{\one}{\mathbf 1}
\newcommand{\Tr}{\operatorname{Tr}}
\newcommand{\erfc}{\operatorname{erfc}}
\newcommand{\Pois}{\operatorname{Pois}}
\newcommand{\Gum}{\operatorname{Gum}}
\title{\bfseries Largest bulk gap of the complex Ginibre ensemble}
\author{Philippe Moreillon}
\date{}

\begin{document}
\maketitle

\begin{abstract}
Let \(M_n(B)\) be the largest distance from an eigenvalue of an \(n\times n\)
complex Ginibre matrix, with entries of variance \(1/n\), lying in a fixed
bulk set \(B\) compactly contained in the unit disk and of planar area
\(|B|\), to its nearest other eigenvalue. Lopatto and Otto proved that $
\sqrt{n} M_n(B)/(4\log n)^{1/4} \to 1$
in probability. Here we prove that $\beta_n^{3/4}\bigl(\sqrt n\,M_n(B)-\beta_n^{1/4}\bigr)$
converges in distribution to a Gumbel random variable, and we determine
\(\beta_n\) explicitly.
\end{abstract}


\section{Introduction and main results}

\subsection{The complex Ginibre ensemble}

The complex Ginibre ensemble is a non-Hermitian random matrix model.
Introduced by Ginibre in 1965 \cite{Ginibre}, it consists of matrices with
independent centered complex Gaussian entries of variance \(1/n\).  We refer to the monograph of Byun
and Forrester \cite{ByunForrester} for an account of recent developments on the subject.  

The unordered complex eigenvalues \(z_1,\ldots,z_n\) have joint probability measure
\begin{equation}\label{eq:introjpdf}
 \frac{n^{n(n+1)/2}}{\pi^n\prod_{j=1}^n j!}
 \exp\!\left(-n\sum_{j=1}^n|z_j|^2\right)
 \prod_{1\le i<j\le n}|z_i-z_j|^2
 \prod_{j=1}^n\dd^2z_j.
\end{equation}
As is well-known, the associated point process is determinantal, and as $n\to\infty$, the empirical eigenvalue measure converges to the uniform probability
measure on the unit disk. Thus a typical bulk nearest-neighbour distance is of order
\(n^{-1/2}\).  The present paper concerns a substantially more delicate
question: how large can such a distance be among all eigenvalues lying
in a prescribed portion of the bulk?

\subsection{Extreme spacings in one and two dimensions}

The smallest gaps between nearest-neighbour eigenvalues arise from exceptionally
close pairs and therefore require detailed knowledge of the asymptotic behavior of the correlation
kernel. The largest gaps, by contrast, are created by exceptionally isolated eigenvalues and are governed by a large-deviation event: a disk centred at such an eigenvalue, with radius slightly larger than \(n^{-1/2}\), must contain no other eigenvalue. The precise large-\(n\) asymptotics of this hole probability were recently obtained in \cite{CharlierAnnuli}, and we will rely on this result.

The study of extreme eigenvalue spacings was initiated by Vinson \cite{V2001} and now has a substantial history, particularly in dimension one.  For the CUE and GUE, Ben Arous and Bourgade
\cite{BenArousBourgade} obtained the joint limiting laws of the smallest gaps
and the leading-order term of the largest gaps.  Soshnikov's earlier work
\cite{Soshnikov} established Poisson statistics for the smallest spacings of a
broad class of translation-invariant determinantal processes.  The largest
gaps were subsequently resolved at the fluctuation scale by Feng and Wei
\cite{FengWei}: after a precise centering, the largest gaps of the
CUE and GUE form a Poisson process and the ordered gaps have generalized
Gumbel laws.  More recently, Charlier \cite{CharlierLargest} obtained sharp
largest-gap asymptotics for general unitary-invariant Hermitian ensembles. 

In dimension two, the theory is much younger. Shi and Jiang
\cite{ShiJiang} proved that the smallest gaps of the complex Ginibre ensemble
are of order \(n^{-3/4}\) and converge, after
rescaling, to Poisson statistics. This result was extended by Charlier \cite{CharlierSmallest} to general random normal matrices. The same scale and limiting Poisson
statistics were obtained by Lopatto and Meeker \cite{LopattoMeeker} for the
non-real bulk eigenvalues of the real Ginibre ensemble. Related Poisson
limits for the smallest distances between zeros of Gaussian analytic
functions on compact Riemann surfaces, together with an analogue for the
planar Gaussian entire function, were established by Feng and Yao
\cite{FengYao}. On the opposite side of the spacing spectrum, Otto
\cite{Otto} proved a Poisson limit for points with unusually large
nearest-neighbour distance in the infinite Ginibre process observed through
expanding windows, and deduced the leading scale of the maximum. Lopatto and
Otto \cite{LopattoOtto} then initiated the study of the
largest bulk nearest-neighbour gap in the complex Ginibre ensemble.  Writing
\(M_n(B)\) for the largest nearest-neighbour distance among eigenvalues
anchored over a fixed bulk set \(B\), as made precise in \eqref{eq:Mn}, they
proved that
\begin{equation}\label{eq:LOleading}
 \frac{\sqrt{n}M_n(B)}{(4\log n)^{1/4}}\longrightarrow1
 \qquad\text{in probability}.
\end{equation}
Their work also provided a source of inspiration for the present work.

\subsection{Repulsion, hole probabilities, and the two extreme scales}

As mentioned earlier, the smallest and largest
nearest-neighbour distances in the Ginibre ensemble \eqref{eq:introjpdf} occupy the very different scales
\[
 n^{-3/4}
 \qquad\hbox{and}\qquad
 n^{-1/2}(4\log n)^{1/4},
\]
respectively, and obtaining precise information on the fluctuations of the largest gap requires asymptotics for the hole probability on a disk of radius larger than \(n^{-1/2}\).

Disk-hole probabilities for the Ginibre ensemble have a long history. Early contributions to this problem include Forrester \cite{ForresterHole} and Jancovici,
Lebowitz and Manificat \cite{JLM}. Adhikari and Reddy \cite{AdhikariReddy} and Charlier \cite{CharlierBalayage} studied hole probabilities for general shapes at leading order. For the much finer precision required here, Charlier's analysis of disk counting probabilities \cite{CharlierAnnuli} provides the precise expansion, including the order-one
term, that underlies the analytic part of the present work: as \(n\to\infty\) with \(r\) fixed, he obtained
\begin{align}
\mathbb{P}(\#\{|z_{j}|<r\}=0)
=
\exp\left(
C_1n^2+C_2n\log n+C_3n+C_4\sqrt n
+C_5\log n+C_6 +O(n^{-1/12})
\right),
\label{general circular Ginibre hole expansion}
\end{align}
where
\begin{align*}
& C_1=-\frac{r^4}{4},
\quad
C_2 =-\frac{r^2}{2}, \quad C_3=r^2\left(1-\log(r\sqrt{2\pi})\right), \\
& C_4=a_{1}r, \quad
C_5=\frac13, \quad C_6
=
a_{0}
+\frac23 \log r ,
\end{align*}
and \(a_{1},a_{0}\) are defined in \eqref{eq:a1} and \eqref{eq:a0}. We also mention that a result analogous to \eqref{general circular Ginibre hole expansion} for the spherical ensemble is available in \cite{BP2026}.

\subsection{The extreme-value mechanism and the contribution of this paper}

Our aim is to pass from the leading law \eqref{eq:LOleading} to a complete
extreme-value theorem.  We identify the canonical tail transform, prove a
Gumbel limit for the largest gap, and determine every deterministic term in
the fourth-power centering down to order one.  We also obtain the limiting
laws of all fixed extreme order statistics.  The answer is most naturally
stated in terms of the reduced-Palm hole probability of the infinite Ginibre
process. We write \(\xi_\infty^{0,!}\) for the process governed by the
reduced Palm distribution at the origin, with the distinguished Palm point
removed. If a point is fixed at the centre of a disk, the microscopic event
that its nearest neighbour lies farther than \(R\) is
\[
 q(R)=\Pp\{\xi_\infty^{0,!}(D_R)=0\}.
\]
Consequently the correct analogue of the exponential tail transform for
independent random variables is \(\Psi(R)=-\log q(R)\). At a physical radius
\(r\), the expected number of anchors with an empty surrounding disk is
asymptotic to
\[
 \frac{n|B|}{\pi}q(\sqrt n\,r).
\]
Choosing \(r\) by
\(\Psi(\sqrt n\,r)=\log(n|B|/\pi)+x\) therefore makes this mean tend
to \(e^{-x}\), suggesting both the Poisson law for exceedances and the Gumbel
law for the maximum.  Theorem~\ref{thm:main} makes this heuristic exact.

Two features are specific to the required precision.  First, the disk is
\emph{anchored at an eigenvalue}.  The relevant quantity is therefore the
reduced-Palm hole probability \(q\), rather than the ordinary Ginibre hole
probability \(H\).  For Ginibre the exact identity
\[
 q(R)=e^{R^2}H(R)
\]
shows that confusing the two leaves the leading \(R^4/4\) term unchanged but
alters the entire \(R^2\) correction. It is thus invisible in
\eqref{eq:LOleading} and fatal at Gumbel accuracy.  Second, one needs the
large-hole expansion with an \(o(1)\) error after inserting the growing
microscopic radius \(R\asymp(\log n)^{1/4}\). A fixed-parameter asymptotic
cannot simply be evaluated along this regime.  We derive the necessary
compact-uniform form of Charlier's disk-hole analysis
\cite{CharlierAnnuli}, and then compare the finite and infinite reduced-Palm
determinants with exponentially small error.

The result exhibits a useful dimensional contrast.  In the sine process,
the logarithm of a large interval probability is quadratic in its length;
this produces largest one-dimensional bulk gaps on the scale
\(\sqrt{\log n}\) times the typical spacing.  For the infinite Ginibre
process, the logarithm of a large disk probability begins with \(-R^4/4\);
the maximal nearest-neighbour distance is therefore of order
\((\log n)^{1/4}\) times the typical two-dimensional spacing.  Beyond this
leading exponent, the terms \(R^2\log R\), \(R^2\), \(R\), \(\log R\), and
the constant term all survive, after inversion, at successively smaller but
still deterministic orders.  Keeping this entire hierarchy is what yields
the explicit centering in Theorem~\ref{thm:explicit}.

Let us summarize the main contributions of the paper.  First, we express the
limiting law in the exact tail coordinate \(\Psi\), which removes all
ambiguity from the normalization.  Second, we prove a sharp expansion of the
Ginibre reduced-Palm hole probability, including explicit convergent integral
formulae for the linear and constant coefficients.  Third, we transfer the
adaptive-radius Poisson approximation of \cite{LopattoOtto} to the single
deterministic radius selected by \(q\), uniformly throughout a fixed bulk set.
Finally, we invert the hole expansion and obtain a fourth-power normalization
whose deterministic centering is accurate through order one.

The exact tail coordinate also clarifies which part of the argument is
probabilistic and which part is analytic.  The probabilistic input asserts
that eigenvalues isolated at an adaptively chosen radius form, asymptotically,
a Poisson process.  The analytic input identifies that adaptive radius with a
single deterministic radius, uniformly over the prescribed bulk set, and
then evaluates it sharply.  The latter step requires substantially more than
the leading large-hole rate. An error of size \(o(R^2)\), for example, is
enough to recover \eqref{eq:LOleading} but is much too large to locate a
Gumbel fluctuation, whose change in the logarithmic tail is of order one.
Our expansion retains all terms that remain visible after this inversion.

\subsection{Organization and notation}

The paper is organized as follows.  The remainder of this section fixes the
normalization and states the principal results.  Section~2 recalls the finite
and infinite Ginibre kernels, identifies the reduced-Palm product, and records
the Poisson input.  Section~3 proves the full reduced-Palm hole expansion.
Section~4 compares the finite and infinite reduced-Palm determinants uniformly
for physical radii \(r=O(n^{-1/2}(\log n)^{1/4})\). Section~5 passes from the adaptive Poisson theorem
to a common deterministic threshold and inverts the tail asymptotics.  The
compact-uniform incomplete-gamma summation is proved in Appendix~A.

Let \(X_n=(x_{ij})_{i,j=1}^n\), where the \(x_{ij}\) are independent centered
complex Gaussian variables of variance \(1/n\), each with density
\((n/\pi)e^{-n|z|^2}\,\mathrm d^2z\), and write \(\xi_n\) for its eigenvalue
point process. Thus the circular-law droplet is the unit disk, and the
bulk intensity in these coordinates is asymptotic to \(n/\pi\). For \(a\in\C\) and
\(r\ge0\), write
\[
 D_r(a)=\{z\in\C:|z-a|<r\},\qquad D_r=D_r(0),
\]
and write \(|A|\) for the planar Lebesgue measure of a Borel set \(A\).
Throughout,
\(B\subset\C\) is a fixed bounded Borel set such that
\begin{equation}\label{eq:Bassumption}
 0<|B|<\infty,\qquad |\partial B|=0,
 \qquad \overline B\subset\{z:|z|<1\}.
\end{equation}
Define the anchored nearest-neighbour gaps and their maximum by
\begin{equation}\label{eq:Mn}
 d_n(z)=\min_{w\in\xi_n\setminus\{z\}}|z-w|,
 \qquad
 M_n(B)=\max_{z\in\xi_n\cap B}d_n(z),
\end{equation}
with the maximum over the empty set defined as \(0\). Neighbours are
taken from the full spectrum, not merely from \(B\).

Let \(\xi_\infty\) denote the infinite Ginibre process of intensity
\(1/\pi\), and let \(\xi_\infty^{0,!}\) denote its reduced Palm process at
the origin, with the distinguished Palm point removed. The probability that a
point fixed at the origin has no other point within distance \(r\) is
\begin{equation}\label{eq:qdefinition}
 q(r):=\Pp\{\xi_\infty^{0,!}(D_r)=0\}
 =\prod_{k=2}^{\infty}\frac{\Gamma(k,r^2)}{\Gamma(k)},
 \qquad
 \Gamma(k,x)=\int_x^\infty t^{k-1}e^{-t}\dd t.
\end{equation}
The product converges locally uniformly, and \(q\) is continuous and
strictly decreasing from \(1\) to \(0\); see \Cref{lem:palmproduct}.
Consequently
\[
 \Psi(r):=-\log q(r),\qquad r\ge0,
\]
is a continuous strictly increasing bijection of \([0,\infty)\) onto
itself.

\begin{theorem}[Exact tail-transform normalization]\label{thm:main}
Under \eqref{eq:Bassumption}, for every \(x\in\R\),
\begin{equation}\label{eq:mainlimit}
 \lim_{n\to\infty}
 \Pp\!\left\{
   \Psi(\sqrt n\,M_n(B))-\log\frac{n|B|}{\pi}\le x
 \right\}
 =\exp\{-e^{-x}\}.
\end{equation}
\end{theorem}

To extract an explicit centering from \Cref{thm:main}, we first determine the
reduced-Palm hole probability sharply. Set
\begin{align}
 a_2&=2-\frac12\log(2\pi),\label{eq:a2}\\
 a_1&=\sqrt2\left\{
 \int_{-\infty}^{0}\log\!\left(\frac{\erfc y}{2}\right)\dd y
 +\int_0^\infty\left[
 \log\!\left(\frac{\erfc y}{2}\right)+y^2+\log y
 +\log(2\sqrt\pi)\right]\dd y\right\},\label{eq:a1}
\end{align}
and define
\begin{align}
 J_-={}&\int_{-\infty}^{0}\left[
 2y\log\!\left(\frac{\erfc y}{2}\right)
 +\frac{e^{-y^2}(1-5y^2)}{3\sqrt\pi\,\erfc y}\right]\dd y,\label{eq:Jminus}\\
 J_+={}&\int_{0}^{\infty}\left[
 2y\log\!\left(\frac{\erfc y}{2}\right)
 +\frac{e^{-y^2}(1-5y^2)}{3\sqrt\pi\,\erfc y}
 +\frac{11}{3}y^3+2y\log y
 +\left(\frac12+2\log(2\sqrt\pi)\right)y\right]\dd y,\label{eq:Jplus}\\
 a_0={}&\frac14\log(2\pi)-\zeta'(-1)-2J_--2J_+ .\label{eq:a0}
\end{align}
Here \(\erfc y=(2/\sqrt\pi)\int_y^\infty e^{-t^2}\,\mathrm dt\), and
\(\zeta\) denotes the Riemann zeta function. All four integrals converge
absolutely after the displayed renormalizations. Numerical quadrature of
these convergent integrals gives
\[
 a_1=-2.3014159133\ldots,\qquad a_0=1.0063109967\ldots.
\]

\begin{theorem}[Sharp reduced-Palm hole asymptotics]\label{thm:palmexpansion}
As \(r\to\infty\),
\begin{equation}\label{eq:palmexpansion}
 \log q(r)=-\frac{r^4}{4}-r^2\log r+a_2r^2+a_1r
 +\frac23\log r+a_0+O(r^{-1/6}).
\end{equation}
\end{theorem}

Inverting \eqref{eq:palmexpansion} yields an explicit fourth-power centering.
For all sufficiently large \(n\), define
\begin{align}
 \beta_n={}&4\log\frac{n|B|}{\pi}
 -4\sqrt{\log\frac{n|B|}{\pi}}\,
   \log\!\left(2\sqrt{\log\frac{n|B|}{\pi}}\right)
 +8a_2\sqrt{\log\frac{n|B|}{\pi}}\notag\\
 &+4\sqrt2\,a_1\left(\log\frac{n|B|}{\pi}\right)^{1/4}
 +2\left[\log\!\left(2\sqrt{\log\frac{n|B|}{\pi}}\right)\right]^2
 \notag\\
 &+\left(\frac{10}{3}-8a_2\right)
   \log\!\left(2\sqrt{\log\frac{n|B|}{\pi}}\right)
 +8a_2^2-4a_2+4a_0.\label{eq:beta}
\end{align}

\begin{theorem}[Explicit fourth-power normalization]
\label{thm:explicit}
Under \eqref{eq:Bassumption}, for every \(x\in\R\),
\begin{equation}\label{eq:fourth}
 \Pp\{n^2M_n(B)^4\le\beta_n+4x\}\longrightarrow e^{-e^{-x}}.
\end{equation}
Equivalently,
\begin{equation}\label{eq:affine}
 \beta_n^{3/4}\bigl(\sqrt n\,M_n(B)-\beta_n^{1/4}\bigr)
 \ \Longrightarrow\ \Gum.
\end{equation}
\end{theorem}

The square of the microscopic centering radius \(\beta_n^{1/4}\) for
\(\sqrt n\,M_n(B)\) in \eqref{eq:affine} has the simpler
large-\(n\) expansion
\[
 \beta_n^{1/2}
 =2\sqrt{\log\frac{n|B|}{\pi}}
  -\frac12\log\log\frac{n|B|}{\pi}
  +4-\log(4\pi)+o(1).
\]
Consequently the largest gap has the readable location formula
\begin{equation}\label{eq:simplecentering}
 nM_n(B)^2
 =2\sqrt{\log\frac{n|B|}{\pi}}
  -\frac12\log\log\frac{n|B|}{\pi}
  +4-\log(4\pi)+o_{\Pp}(1).
\end{equation}
Formula \eqref{eq:beta} retains every deterministic correction through order
one. In decreasing order, their sizes are
\[
 \sqrt{\log n}\log\log n,\quad \sqrt{\log n},\quad(\log n)^{1/4},
 \quad(\log\log n)^2,\quad\log\log n,\quad1.
\]

\begin{corollary}[Fixed order statistics]\label{cor:kth}
Arrange the values \(\{d_n(z):z\in\xi_n\cap B\}\) in decreasing order
as \(M_{n,1}\ge M_{n,2}\ge\cdots\), and put \(M_{n,k}=0\) when fewer than
\(k\) anchors lie in \(B\). For each fixed \(k\ge1\),
\[
 \Pp\!\left\{
   \Psi(\sqrt n\,M_{n,k})-\log\frac{n|B|}{\pi}\le x
 \right\}
 \longrightarrow e^{-e^{-x}}\sum_{j=0}^{k-1}\frac{e^{-jx}}{j!}.
\]
\end{corollary}

\section{Ginibre kernels, Palm holes, and the Poisson input}

\subsection{From the eigenvalue density to the projection kernel}

We collect the determinantal facts used later and keep track of the
normalization. If \(X_n\) has density
\((n/\pi)^{n^2}e^{-n\Tr X_nX_n^*}\) on \(\C^{n\times n}\), the complex Schur
decomposition, followed by integration over the strictly upper triangular
entries and the unitary factor, gives \eqref{eq:introjpdf}; see
\cite[Chapter~2]{ByunForrester}.  Since
\[
 \prod_{1\le i<j\le n}|z_i-z_j|^2
 =\left|\det[z_i^{j-1}]_{i,j=1}^n\right|^2,
\]
the density is an orthogonal-polynomial ensemble for the planar Gaussian
weight.  The monomials satisfy
\begin{equation}\label{eq:monomialorthogonality}
 \int_{\C}z^j\bar z^k e^{-n|z|^2}\dd^2z
 =\frac{\pi j!}{n^{j+1}}\,\one_{\{j=k\}},
\end{equation}
as follows by passing to polar coordinates.  Thus the functions
\(\varphi_{j,n}\) below form an orthonormal family in
\(L^2(\C,\dd^2z)\).

For completeness, recall that the \(k\)-point correlation function
\(\rho_n^{(k)}\) is characterized by
\begin{align}\label{eq:correlationdefinition}
 &\mathbb E\!\sum_{z_1,\ldots,z_k\in\xi_n}^{\ne}
 f(z_1,\ldots,z_k)=
 \int_{\C^k}f(z_1,\ldots,z_k)\rho_n^{(k)}(z_1,\ldots,z_k)
 \prod_{j=1}^k\dd^2z_j
\end{align}
for every nonnegative measurable \(f\), where the sum is over ordered
tuples of distinct points.  AndrÃ©ief's identity applied to
\eqref{eq:introjpdf} and \eqref{eq:monomialorthogonality} yields
\[
 \rho_n^{(k)}(z_1,\ldots,z_k)
 =\det[K_n(z_i,z_j)]_{i,j=1}^k.
\]
With respect to planar Lebesgue measure, \(\xi_n\) is determinantal with
kernel \cite[Theorem 4.3.10]{HKPV}
\begin{equation}\label{eq:Kn}
 \begin{aligned}
 \varphi_{j,n}(z)&=\left(\frac{n^{j+1}}{\pi j!}\right)^{1/2}
 z^j e^{-n|z|^2/2},\\
 K_n(z,w)&=\sum_{j=0}^{n-1}\varphi_{j,n}(z)\overline{\varphi_{j,n}(w)}=\frac n\pi e^{-\frac n2(|z|^2+|w|^2)}
   \sum_{j=0}^{n-1}\frac{(nz\bar w)^j}{j!}.
 \end{aligned}
\end{equation}
The microscopic infinite Ginibre kernel is the locally trace-class projection
\begin{equation}\label{eq:Kinf}
 K_\infty(z,w)=\sum_{j=0}^{\infty}\varphi_j(z)\overline{\varphi_j(w)}
 =\frac1\pi e^{z\bar w-|z|^2/2-|w|^2/2}.
\end{equation}
Here \(\varphi_j(z)=z^je^{-|z|^2/2}/\sqrt{\pi j!}\), and
\(\xi_\infty\) denotes the determinantal process with kernel \(K_\infty\).
Its unit-disk-scale
version, of intensity \(n/\pi\), is
\begin{equation}\label{eq:Kinfn}
 K_{\infty,n}(z,w)=nK_\infty(\sqrt n\,z,\sqrt n\,w)
 =\frac n\pi e^{nz\bar w-n|z|^2/2-n|w|^2/2}.
\end{equation}
The diagonal of the finite kernel has the useful Poisson representation
\begin{equation}\label{eq:densitypoisson}
 K_n(z,z)=\frac n\pi e^{-n|z|^2}\sum_{j=0}^{n-1}\frac{(n|z|^2)^j}{j!}
 =\frac n\pi\Pp\{\Pois(n|z|^2)\le n-1\}.
\end{equation}
In particular, if \(|z|\le s\) with fixed \(s<1\), a Chernoff
bound gives
\begin{equation}\label{eq:bulkintensityprelim}
 K_n(z,z)=\frac n\pi\bigl(1+O_s(e^{-c_sn})\bigr).
\end{equation}
This is the quantitative bulk flatness that later makes the deterministic
threshold independent of the location of the anchor.

Although the expression \eqref{eq:Kinf} is not literally translation
invariant, its associated point process is.  Indeed, for \(a\in\C\), the
magnetic translation
\[
 (U_af)(z)=e^{i\operatorname{Im}(z\bar a)}f(z-a)
\]
is unitary on \(L^2(\C,\dd^2z)\).  The identity
\[
 K_\infty(z+a,w+a)
 =e^{i\operatorname{Im}(z\bar a)}K_\infty(z,w)
  e^{-i\operatorname{Im}(w\bar a)}
\]
shows that it conjugates the compression of \(K_\infty\) to \(D_r(0)\)
with that to \(D_r(a)\).
Fredholm determinants are invariant under this conjugation.  Consequently
both ordinary and reduced-Palm disk-hole probabilities depend only on the
radius, once the Palm point and disk centre are translated together.
The same statement holds for \(K_{\infty,n}\), with magnetic phase
\(e^{in\operatorname{Im}(z\bar a)}\). Dilation by \(\sqrt n\) identifies
its disk of physical radius \(r\) with a disk of microscopic radius
\(\sqrt n\,r\) for \(K_\infty\).

\subsection{Kostlan's radial representation}

The rotational symmetry of Ginibre gives an additional exact structure that
is particularly useful for disk events.  The following result is usually
referred to as Kostlan's theorem \cite{Kostlan}; see also
\cite[Section~2.9]{ByunForrester}.

For \(k\ge1\) and \(x\ge0\), write
\begin{equation}\label{eq:Qdefinition}
 Q(k,x)=\frac{\Gamma(k,x)}{\Gamma(k)},\qquad
 \Gamma(k,x)=\int_x^\infty t^{k-1}e^{-t}\dd t,qquad
 \gamma(k,x)=\int_0^x t^{k-1}e^{-t}\dd t.
\end{equation}

\begin{lemma}[Kostlan]\label{lem:kostlan}
Let \(z_1,\ldots,z_n\) be the eigenvalues of \(X_n\), with their labels
discarded.  The unordered multiset
\[
 \{n|z_1|^2,\ldots,n|z_n|^2\}
\]
has the same law as the unordered multiset
\(\{\gamma_1,\ldots,\gamma_n\}\), where the variables are independent and
\(\gamma_k\) has gamma density
\[
 \frac{x^{k-1}e^{-x}}{\Gamma(k)}\one_{\{x>0\}}\dd x.
\]
Consequently,
\begin{equation}\label{eq:finiteordinaryhole}
 \Pp\{\xi_n(D_r)=0\}=\prod_{k=1}^nQ(k,nr^2).
\end{equation}
\end{lemma}

\begin{proof}
Write \(z_j=\sqrt{x_j/n}\,e^{i\theta_j}\). Expanding the two Vandermonde
determinants in \eqref{eq:introjpdf} gives
\[
 |\det[z_i^{j-1}]|^2
 =n^{-n(n-1)/2}\sum_{\sigma,\tau\in S_n}\operatorname{sgn}(\sigma\tau)
 \prod_{i=1}^n
 x_i^{(\sigma(i)+\tau(i)-2)/2}
 e^{i(\sigma(i)-\tau(i))\theta_i}.
\]
Integration over \(\theta_1,\ldots,\theta_n\) kills every term except
\(\sigma=\tau\). Since \(\dd^2z_i=(2n)^{-1}\dd x_i\dd\theta_i\), the joint
density of the labeled squared radii is therefore proportional to
\[
 e^{-\sum_i x_i}\operatorname{per}
 [x_i^{j-1}]_{i,j=1}^n\prod_{i=1}^n\one_{\{x_i>0\}}\dd x_i,
\]
where \(\operatorname{per}\) denotes the permanent.  After using the
normalizing factors \(\Gamma(j)\), this is the symmetrization of
\[
 \prod_{j=1}^n\frac{x_j^{j-1}e^{-x_j}}{\Gamma(j)}\dd x_j.
\]
Discarding labels proves the first assertion.  The disk is empty precisely
when every rescaled squared radius exceeds \(nr^2\). Independence then gives
\[
 \prod_{k=1}^n\Pp\{\gamma_k>nr^2\}
 =\prod_{k=1}^n\frac{\Gamma(k,nr^2)}{\Gamma(k)},
\]
which is \eqref{eq:finiteordinaryhole}.
\end{proof}

Kostlan's representation gives a short alternative route to the ordinary
disk product, but not directly to the anchored nearest-neighbour event.
The latter is a statement under a reduced Palm measure, and the Palm
conditioning changes precisely one radial mode.  We therefore return to the
operator formulation.

\subsection{Reduced Palm processes and anchored holes}

For any of the kernels \(K_n,K_{\infty,n},K_\infty\), the reduced-Palm
kernel at \(a\in\C\) is \cite[Theorem 1.7]{ShiraiTakahashi}
\begin{equation}\label{eq:Palmkernel}
 K^a(z,w)=K(z,w)-\frac{K(z,a)K(a,w)}{K(a,a)}.
\end{equation}
To make this projection statement precise, regard the kernels as operators
on \(L^2(\C,\mathrm d^2z)\). The reproducing identity shows that
\eqref{eq:Palmkernel} is the orthogonal projection onto the functions in
the corresponding weighted Fock range that vanish at \(a\). In particular,
\[
 \mathcal H_n=\operatorname{span}(\varphi_{0,n},\ldots,\varphi_{n-1,n}),
 \qquad
 \mathcal H_{\infty,n}=\overline{\operatorname{span}}
 \{\varphi_{j,n}:j\ge0\},
 \qquad \mathcal H_n^a\subset\mathcal H_{\infty,n}^a.
\]

\begin{lemma}[The reduced-Palm product]\label{lem:palmproduct}
For the infinite Ginibre process of intensity \(1/\pi\),
\[
 \Pp\{\xi_\infty^{0,!}(D_r)=0\}
 =\det(I-\one_{D_r}K_\infty^0\one_{D_r})
 =\prod_{k=2}^{\infty}Q(k,r^2)=q(r).
\]
The product converges locally uniformly, and \(q\) is continuous and
strictly decreasing from \(1\) to \(0\).
If \(H(r)=\Pp\{\xi_\infty(D_r)=0\}\), then
\begin{equation}\label{eq:palmordinary}
 q(r)=e^{r^2}H(r).
\end{equation}
This identity is also recorded in
\cite[Remark~3.1(1)]{ByunForrester}.
\end{lemma}

\begin{proof}
At the origin,
\(K_\infty(\,\cdot\,,0)=\varphi_0/\sqrt\pi\), and the rank-one
subtraction in \eqref{eq:Palmkernel} therefore removes precisely the
constant Fock mode:
\[
 K_\infty^0(z,w)
 =\sum_{j=1}^{\infty}\varphi_j(z)\overline{\varphi_j(w)}.
\]
Let \(T_r=\one_{D_r}K_\infty^0\one_{D_r}\), acting on
\(L^2(D_r,\dd^2z)\).  The vectors \(\one_{D_r}\varphi_j\), \(j\ge1\),
are mutually orthogonal by angular integration, and the nonzero
eigenvalues of \(T_r\) are their squared norms.  Passing to polar
coordinates and then setting \(s=\rho^2\) gives
\begin{align*}
 \lambda_j(r)
 &=\|\one_{D_r}\varphi_j\|_2^2
 =\frac{2}{j!}\int_0^r\rho^{2j+1}e^{-\rho^2}\dd\rho=\frac1{j!}\int_0^{r^2}s^je^{-s}\dd s
 =\frac{\gamma(j+1,r^2)}{j!}.
\end{align*}
Moreover,
\[
 \Tr T_r=\int_{D_r}K_\infty^0(z,z)\dd^2z<\infty,
\]
so the eigenvalues are summable and the Fredholm determinant is the
absolutely convergent product
\[
 \det(I-T_r)
 =\prod_{j=1}^{\infty}
 \left(1-\frac{\gamma(j+1,r^2)}{j!}\right)
 =\prod_{k=2}^{\infty}Q(k,r^2).
\]
For every \(R<\infty\), the gamma--Poisson identity gives
\[
 \sup_{0\le r\le R}\sum_{k>K}\bigl(1-Q(k,r^2)\bigr)
 \le \sum_{k>K}\Pp\{\Pois(R^2)\ge k\}\longrightarrow0.
\]
Thus the product converges locally uniformly and is positive. Each factor is
strictly decreasing for \(r>0\), while
\(q(0)=1\) and
\(q(r)\le Q(2,r^2)=e^{-r^2}(1+r^2)\to0\). This proves the stated
continuity and monotonicity.
For the ordinary process the \(j=0\) mode is also present.  Its factor in
the determinant is
\[
 1-\gamma(1,r^2)=\Gamma(1,r^2)=e^{-r^2}.
\]
Thus \(H(r)=e^{-r^2}q(r)\), which is equivalent to
\eqref{eq:palmordinary}.
\end{proof}

Under the dilation \(z\mapsto z/\sqrt n\), the intensity-\(1/\pi\)
process \(\xi_\infty\) becomes the process with kernel
\(K_{\infty,n}\). Hence its reduced-Palm hole probability in a physical
disk of radius \(r\) is
\begin{equation}\label{eq:scaledpalmproduct}
 \det(I-\one_{D_r(a)}K_{\infty,n}^a\one_{D_r(a)})
 =q(\sqrt n\,r)=\prod_{k=2}^{\infty}Q(k,nr^2).
\end{equation}

Returning to the microscopic radius \(R\), the omission of the \(k=1\)
factor in \(q\) has a direct probabilistic
meaning.  Under the ordinary process, the constant Fock mode contributes the
eigenvalue \(1-e^{-R^2}\) to the disk compression. Conditioning on a point
at the centre and then removing that point deletes exactly this mode.  Thus
the quotient of the two hole probabilities is the reciprocal of
\(Q(1,R^2)\), which is \(e^{R^2}\). This factor is negligible compared with
the leading exponent \(R^4/4\), but it is one of the dominant corrections at
the fluctuation precision considered here.

We shall repeatedly use the Fredholm-determinant interpretation in the
following elementary form.  If \(K\) is a locally trace-class Hermitian
contraction defining a determinantal process \(\eta\), then for every bounded
Borel set \(D\),
\begin{equation}\label{eq:genericvoiddet}
 \Pp\{\eta(D)=0\}=\det(I-\one_DK\one_D).
\end{equation}
Indeed, the number of points in \(D\) is distributed as a sum of independent
Bernoulli variables with parameters equal to the eigenvalues of the
compression.  Formula \eqref{eq:genericvoiddet} follows by multiplying their
zero probabilities.  This also shows directly that enlarging \(D\) decreases
the void probability.

\begin{lemma}[Finite Palm-hole monotonicity]\label{lem:qnmonotone}
For \(n\ge2\) and \(a\in\C\), the function
\[
 q_n(a,r):=\Pp\{\xi_n^{a,!}(D_r(a))=0\},\qquad r\ge0,
\]
is continuous and strictly decreasing from \(1\) to \(0\). Moreover,
\((a,r)\mapsto K_n(a,a)q_n(a,r)\) is continuous.
\end{lemma}

\begin{proof}
Write \(Z_{n,a}\) for the normalizing constant of the reduced-Palm density.
The correlation-density formula, with one eigenvalue fixed at \(a\), gives
\begin{align}\label{eq:finitePalmholeintegral}
 q_n(a,r)=\frac1{Z_{n,a}}
 \int_{\{|z_i-a|\ge r,\ 1\le i\le n-1\}}
 &e^{-n\sum_{i=1}^{n-1}|z_i|^2}
 \prod_{i=1}^{n-1}|z_i-a|^2\notag\\
 &\times\prod_{1\le i<j\le n-1}|z_i-z_j|^2
 \prod_{i=1}^{n-1}\mathrm d^2z_i.
\end{align}
Thus, up to the positive normalizing constant \(Z_{n,a}\), the reduced-Palm
law has density
\begin{equation}\label{eq:finitePalmdensity}
 e^{-n\sum_{i=1}^{n-1}|z_i|^2}
 \prod_{i=1}^{n-1}|z_i-a|^2
 \prod_{1\le i<j\le n-1}|z_i-z_j|^2
 \prod_{i=1}^{n-1}\mathrm d^2z_i .
\end{equation}
The integrand is strictly positive away from the algebraic set on which one
coordinate equals \(a\) or two coordinates coincide; that exceptional set
has Lebesgue measure zero.

Let \(0\le r_1<r_2\).  Choose a nonempty open ball contained in
\(D_{r_2}(a)\setminus\overline{D_{r_1}(a)}\), and choose \(n-2\) pairwise
disjoint open balls outside \(\overline{D_{r_2}(a)}\), also disjoint from
the first ball.  Configurations with one coordinate in the first ball and
the remaining coordinates in the other balls have positive reduced-Palm
probability.  Every such configuration avoids \(D_{r_1}(a)\) but not
\(D_{r_2}(a)\), proving strict decrease.

If \(r_m\to r\), the indicators in \eqref{eq:finitePalmholeintegral}
converge pointwise except when one of the coordinates belongs to
\(\partial D_r(a)\).  This exceptional set is null, and dominated
convergence proves continuity in \(r\).  Clearly \(q_n(a,0)=1\).  As
\(r\to\infty\), the void events decrease to the event that the
reduced-Palm configuration, which has exactly \(n-1\) points, is empty;
the latter event has probability zero.  Hence the range is all of
\((0,1]\), with limit zero at infinity.

For joint continuity, set \(z_i=a+u_i\) in both the numerator and the
normalizing integral.  On a compact set of values of \(a\), the translated
integrand is bounded by an integrable function of the form
\[
 C\exp\!\left(-\frac n2\sum_i|u_i|^2\right)
 \left(1+\sum_i|u_i|\right)^{C_n}.
\]
If \((a_m,r_m)\to(a,r)\), dominated convergence applies to the numerator;
the moving circle boundaries are again null.  It also applies to the
normalizing denominator.  Therefore \((a,r)\mapsto q_n(a,r)\) is
continuous.  Since \(K_n(a,a)>0\) and is continuous, the last assertion
follows.
\end{proof}

\subsection{Rare isolated points and their intensity}

For a deterministic radius \(r\), introduce the exceedance process
\begin{equation}\label{eq:exceedanceprocess}
 \mathcal E_n(r)=
 \sum_{a\in\xi_n\cap D_1}
 \one_{\{(\xi_n\setminus\{a\})(D_r(a))=0\}}\,\delta_a.
\end{equation}

\begin{lemma}[Campbell--Palm intensity identity]
\label{lem:campbellpalm}
For every Borel set \(A\subset D_1\),
\begin{equation}\label{eq:exceedanceintensity}
 \mathbb E\mathcal E_n(r)(A)
 =\int_AK_n(a,a)q_n(a,r)\dd^2a.
\end{equation}
More generally, the same formula holds with \(r\) replaced by a measurable
location-dependent radius \(r(a)\).  In particular, the adaptive process in
\Cref{prop:LO} has intensity measure exactly
\(\kappa\dd^2z\) in unit-disk coordinates.
\end{lemma}

\begin{proof}
Apply the reduced Campbell--Mecke formula to
\[
 f(a,\omega)=\one_{\{a\in A\}}
 \one_{\{\omega(D_{r(a)}(a))=0\}}.
\]
The left-hand side is the expected number of retained anchors in \(A\),
and the reduced-Palm expectation on the right-hand side is
\(q_n(a,r(a))\).  This gives the first integral in
\eqref{eq:exceedanceintensity}. If \(r(a)=R_{n,\kappa}(a)\), equation
\eqref{eq:adaptive} makes the integrand equal to \(\kappa\), proving
the final assertion.
\end{proof}

Thus \(K_n(a,a)q_n(a,r)\) is the intensity density of
anchors whose nearest-neighbour distance exceeds \(r\).  In the bulk,
\eqref{eq:bulkintensityprelim} and the finite/infinite comparison proved in
Section~4 reduce it to \((n/\pi)q(\sqrt n\,r)\). This calculation explains both the
factor \(n|B|/\pi\) in Theorem~\ref{thm:main} and why the reduced-Palm,
rather than ordinary, hole probability is the relevant tail.

For a chosen target intensity \(\kappa>0\), Lopatto and Otto select the
radius separately at each anchor so that the integrand in
\eqref{eq:exceedanceintensity} is exactly
\(\kappa\).  The preceding monotonicity lemma guarantees that this adaptive
radius is unambiguously defined. Indeed, uniformly for \(|a|<1\),
\[
 K_n(a,a)\ge\frac n\pi\Pp\{\Pois(n)\le n-1\}
 =\frac n{2\pi}(1+o(1)),
\]
where the inequality follows from \eqref{eq:densitypoisson} and monotonicity
of the Poisson distribution function. Thus \(K_n(a,a)>\kappa\) for all
large \(n\); since \(q_n(a,\,\cdot\,)\) decreases continuously from \(1\)
to \(0\), the defining equation below has a unique solution.

We use the following theorem of Lopatto--Otto
\cite[Theorem 1]{LopattoOtto}.
\begin{proposition}[Lopatto--Otto Poisson approximation]\label{prop:LO}
Fix \(\kappa>0\), \(\varepsilon\in(0,1/16)\), and a Borel set
\(A\subset\{z:|z|<1\}\) with \(\sup_{z\in A}|z|<1\). For all large \(n\), let
\(R_{n,\kappa}(a)\) be the unique radius satisfying
\begin{equation}\label{eq:adaptive}
 K_n(a,a)\,\Pp\{\xi_n^{a,!}(D_{R_{n,\kappa}(a)}(a))=0\}
 =\kappa.
\end{equation}
Let \(\Xi_{n,\kappa}\) be the point process of the retained unit-disk
anchors \(a\), where \(a\in\xi_n\cap D_1\) is retained when
\(d_n(a)>R_{n,\kappa}(a)\). If \(\zeta_\kappa\) is a Poisson process of
intensity \(\kappa\,\mathrm d^2z\), then, for a constant depending on
\(A,\kappa,\varepsilon\),
\begin{equation}\label{eq:LOKR}
 d_{\mathrm{KR}}(\Xi_{n,\kappa}\cap A,
                  \zeta_\kappa\cap A)
 \le C n^{-1/16+\varepsilon}.
\end{equation}
Consequently, if \(N_{n,\kappa}(A)=\Xi_{n,\kappa}(A)\), then
\[
 N_{n,\kappa}(A)\ \Longrightarrow\ \Pois(\kappa|A|).
\]
\end{proposition}

To verify the normalization, let \(\widetilde K_n,\widetilde q_n\), and
\(\widetilde R_{n,\kappa}\) denote the quantities in
\cite[Theorem~1]{LopattoOtto}, which are written in variance-one
coordinates. Under \(z=\sqrt n\,a\),
\[
 K_n(a,a)=n\widetilde K_n(\sqrt n\,a,\sqrt n\,a),\qquad
 q_n(a,r)=\widetilde q_n(\sqrt n\,a,\sqrt n\,r),
\]
and hence
\[
 R_{n,\kappa}(a)=n^{-1/2}
 \widetilde R_{n,\kappa}(\sqrt n\,a).
\]
Thus \(\Xi_{n,\kappa}\) is exactly the push-forward of the process in
\cite[Theorem~1]{LopattoOtto} under \(z\mapsto z/\sqrt n\), and its
intensity equation is \(K_n(a,a)q_n(a,R_{n,\kappa}(a))=\kappa\).
The cited theorem uses closed rather than open disks. This does not change
the process: by Campbell--Palm, the expected number of relevant pairs lying
on \(\partial D_{R_{n,\kappa}(a)}(a)\) is zero, since every reduced-Palm
one-point intensity is absolutely continuous and every circle has zero
planar area.

Here \(d_{\mathrm{KR}}\) is the Kantorovich--Rubinstein distance on finite
point measures, with the total-variation metric used in
\cite{LopattoOtto}.
The stated convergence of the counts follows directly from the metric, not
only from a separate tightness argument.  For every integer \(m\ge0\), the
map
\[
 h_{A,m}(\mu)=\one_{\{\mu(A)\le m\}}
\]
is bounded and \(1\)-Lipschitz for the total-variation metric: if its values
differ, then the two integer-valued masses of \(A\) differ by at least one.
Thus \eqref{eq:LOKR} controls the difference between the corresponding count
distribution functions.

Here and below the versioned citation is deliberately to the archived first
version, arXiv:2501.04611v1.  Theorem~1 of that version states the result for
every fixed \(\kappa>0\), which is the form used here.  The theorem statement
in the later revision is restricted to \(\kappa>1\); we do not invoke that
revised formulation.  This distinction matters because
\(\kappa=e^{-x}/|B|\) ranges over all of \((0,\infty)\) as the Gumbel level
\(x\) varies.  In the first-version proof, \(\kappa\) is fixed before the
estimates are made, and the constants are permitted to depend on it.  Thus
the cited input is pointwise in \(\kappa\), exactly as required by the
fixed-\(x\) statements below; no uniformity as \(\kappa\downarrow0\) is used.
By
\Cref{lem:qnmonotone}, the threshold in \eqref{eq:adaptive} is unique; its
dependence on \(a\) is continuous by monotone inversion, so the retained
process is measurable.

\section{The sharp infinite reduced-Palm hole}

\subsection{Why a full hole expansion is needed}

The radius in this section is microscopic: a unit-disk radius \(s\) is
represented here by \(r=\sqrt n\,s\). Write \(x=r^2\). The exact ordinary hole probability is
\[
 H(r)=\prod_{k=1}^{\infty}Q(k,x),
 \qquad Q(k,x)=\Pp\{\Pois(x)\le k-1\}.
\]
The leading term \(\log H(r)\sim-r^4/4\) is the electrostatic cost of
creating a disk of radius \(r\) in a plasma of intensity \(1/\pi\).  At the
largest-gap scale, \(r^4\asymp\log n\), so this term determines the power
\((\log n)^{1/4}\).  It does not, however, determine a nondegenerate limiting
law.  Changing the limiting Gumbel coordinate by a bounded amount changes
the radius by order \(r^{-3}\), while the successive deterministic terms in
\(\log H(r)\) shift that radius by much larger amounts.  The expansion must
therefore be known through an additive \(o(1)\) in the logarithm.

The product also makes visible the source of the different terms.  When
\(k\) lies well below \(x\), \(Q(k,x)\) is an exponentially small lower tail
of a Poisson variable; summing its large-deviation expansion produces the
terms of orders \(x^2\), \(x\log x\), and \(x\).  The transition window
\(k-x=O(\sqrt x)\) is described by the complementary error function and
produces the \(\sqrt x\) term.  Euler--Maclaurin endpoint corrections and the
matching of the large-deviation and transition regimes produce \(\log x\)
and the constant.  Terms with \(k\) well above \(x\) are exponentially
small.  The purpose of the next lemma and Appendix~A is to implement this
decomposition with an error uniform in the auxiliary disk parameter.

\subsection{A compact-uniform finite product}

The needed asymptotic does not follow merely by substituting a shrinking
radius into a theorem stated for a fixed disk. We first record the precise
compact-uniform statement and prove its uniformity in \Cref{app:uniform}.

\begin{lemma}[Compact-uniform one-disk asymptotics]\label{lem:uniformdisk}
For every compact interval \(I\Subset(0,1)\), uniformly for \(t\in I\),
\begin{align}
 \sum_{k=1}^{N}\log Q(k,Nt^2)
={}&-\frac{t^4}{4}N^2-\frac{t^2}{2}N\log N
 +t^2\bigl(1-\log(t\sqrt{2\pi})\bigr)N\notag\\
 &+a_1t\sqrt N+\frac13\log N+\frac23\log t+a_0
 +O_I(N^{-1/12}).\label{eq:uniformdisk}
\end{align}
\end{lemma}

\begin{proof}[Proof of \Cref{thm:palmexpansion}]
First consider the ordinary hole product
\(H(r)=\prod_{k\ge1}Q(k,r^2)\), and put \(x=r^2\). Choose
\(N=\lceil4x\rceil\) and \(t=(x/N)^{1/2}\). Then \(Nt^2=x\) and \(t\) lies
in a fixed compact subinterval of \((0,1)\). For \(k>N\ge4x\), the integer
gamma--Poisson identity and Chernoff's bound give
\[
 1-Q(k,x)=\Pp\{\Pois(x)\ge k\}
 \le e^{-x}\left(\frac{ex}{k}\right)^k.
\]
Consequently
\begin{equation}\label{eq:tailproduct}
 \sum_{k>N}|\log Q(k,x)|=O(e^{-c x})
\end{equation}
for an absolute \(c>0\). Apply \Cref{lem:uniformdisk} and substitute
\(x=Nt^2\). The cancellation uses
\[
 t\sqrt N=\sqrt x,\qquad
 -\frac{x}{2}\log N-x\log t=-\frac{x}{2}\log x,
 \qquad
 \frac13\log N+\frac23\log t=\frac13\log x.
\]
Thus
\begin{align}
 \log H(r)={}&-\frac{x^2}{4}-\frac{x}{2}\log x
 +\left(1-\frac12\log(2\pi)\right)x
 +a_1\sqrt x+\frac13\log x+a_0+O(x^{-1/12})\notag\\
={}&-\frac{r^4}{4}-r^2\log r
 +\left(1-\frac12\log(2\pi)\right)r^2
 +a_1r+\frac23\log r+a_0+O(r^{-1/6}).\label{eq:Hexpansion}
\end{align}
Finally, \eqref{eq:palmordinary} adds \(r^2\) to \eqref{eq:Hexpansion},
which is precisely \eqref{eq:palmexpansion}.
\end{proof}

\begin{remark}[Ordinary versus anchored holes]\label{rem:holecomparison}
The ordinary and reduced-Palm expansions differ by exactly \(r^2\).  In
particular,
\[
 \log H(r)=-\frac{r^4}{4}-r^2\log r
 +\left(1-\frac12\log(2\pi)\right)r^2
 +a_1r+\frac23\log r+a_0+O(r^{-1/6}),
\]
whereas \(\log q(r)=\log H(r)+r^2\).  Both have the same electrostatic
leading cost and hence the same fourth-root scale, but they yield different
\(\sqrt{\log n}\)-order contributions to the fourth-power centering.  This
is why an ordinary-hole computation cannot be substituted into the
nearest-neighbour problem.
\end{remark}

For use in the local analysis and the deterministic inversion, define
\begin{equation}\label{eq:Phi}
 \Phi(r)=\frac{r^4}{4}+r^2\log r-a_2r^2-a_1r
 -\frac23\log r-a_0.
\end{equation}

\begin{lemma}[Local variation of the transformed tail]
\label{lem:localtailvariation}
Let \(r\to\infty\) and let \(h=h(r)\) satisfy \(h=O(r^{-3})\).  Then
\begin{equation}\label{eq:localtailvariation}
 \Psi(r+h)-\Psi(r)=r^3h+o(1).
\end{equation}
More generally, if \(r_n,s_n\to\infty\), both are of the same order, and
\(\Phi(s_n)-\Phi(r_n)\to y\), then
\(\Psi(s_n)-\Psi(r_n)\to y\).
\end{lemma}

\begin{proof}
By \Cref{thm:palmexpansion} and the definition \eqref{eq:Phi},
\[
 \Psi(s)-\Psi(r)=\Phi(s)-\Phi(r)+O(r^{-1/6}+s^{-1/6})
\]
whenever \(r,s\to\infty\).  This proves the second assertion.  For the
first, Taylor's theorem and
\[
 \Phi'(r)=r^3+2r\log r+(1-2a_2)r-a_1-\frac{2}{3r},
 \qquad \Phi''(r)=3r^2+O(\log r),
\]
give
\[
 \Phi(r+h)-\Phi(r)=\Phi'(r)h+O(r^2h^2)=r^3h+o(1).
\]
Combining the two displays proves \eqref{eq:localtailvariation}.
\end{proof}

The convergence of the integrals in \eqref{eq:a1}--\eqref{eq:Jplus}
follows from \(\erfc(-y)=2-\erfc(y)\) and
\begin{equation}\label{eq:erfctail}
 \erfc(y)=\frac{e^{-y^2}}{\sqrt\pi\,y}
 \left(1-\frac1{2y^2}+\frac3{4y^4}+O(y^{-6})\right),
 \qquad y\to+\infty.
\end{equation}
At \(-\infty\) the integrands are exponentially small; at \(+\infty\), the
two renormalized integrands in \eqref{eq:a1} and \eqref{eq:Jplus} are
respectively \(O(y^{-2})\) and \(O(y^{-3})\).

\section{Uniform comparison of finite and infinite Palm holes}

In this section \(r\) again denotes a physical radius in the unit-disk
coordinates; its microscopic counterpart is \(\sqrt n\,r\).
The Poisson theorem of Lopatto--Otto uses the finite-\(n\) reduced-Palm hole
at a location-dependent radius.  Our tail transform is instead defined by
the stationary infinite process.  We now show that these two holes agree to
relative exponential accuracy throughout the bulk and throughout the entire
range of radii relevant to the maximum.  Relative accuracy is essential:
both determinants themselves are of order \(n^{-1}\), so an absolute
\(o(1)\) estimate would contain no useful information.

\subsection{Two operator estimates}

\begin{lemma}[A relative determinant inequality]
\label{lem:abstractdetcomparison}
Let \(A,C\) be positive trace-class contractions on a Hilbert space, assume
\(0\le A\le C\) and \(\|C\|<1\), and put \(D=C-A\).  Then
\begin{equation}\label{eq:abstractdetcomparison}
 0\le \log\frac{\det(I-A)}{\det(I-C)}
 \le \|(I-C)^{-1}\|\,\Tr D.
\end{equation}
\end{lemma}

\begin{proof}
For \(0\le t\le1\), set \(C_t=C-tD\).  Then
\(0\le C_t\le C\), so \(I-C_t\) is invertible and
\((I-C_t)^{-1}\le(I-C)^{-1}\) in the operator order.  In finite dimension,
Jacobi's formula gives
\[
 \frac{\dd}{\dd t}\log\det(I-C_t)
 =\Tr[(I-C_t)^{-1}D].
\]
The trace is nonnegative, and
\[
 \Tr[(I-C_t)^{-1}D]
 \le\|(I-C)^{-1}\|\Tr D.
\]
Integration from zero to one proves \eqref{eq:abstractdetcomparison} in
finite dimension.  For trace-class operators, apply the finite-dimensional
argument to spectral projections increasing strongly to the identity.
Trace-norm convergence of the compressions, continuity of the Fredholm
determinant, and monotone convergence of the traces pass both inequalities
to the limit.
\end{proof}

\begin{lemma}[Bulk truncation of the Palm projection]
\label{lem:palmtailkernel}
Fix \(s<1\) and \(C_0<\infty\).  There exist \(c,C>0\) such that, uniformly
for \(|a|\le s\), \(0\le r\le C_0n^{-1/2}(\log n)^{1/4}\), and
\(w\in D_r(a)\),
\begin{align}
 0&\le K_{\infty,n}^a(w,w)-K_n^a(w,w)\le Cne^{-cn},
 \label{eq:palmtaildiagonal}\\
 0&\le K_{\infty,n}(a,a)-K_n(a,a)\le Cne^{-cn}.
 \label{eq:untaildiagonal}
\end{align}
Consequently, for \(D=D_r(a)\),
\begin{equation}\label{eq:palmtailtrace}
 \Tr\!\left[\one_D(K_{\infty,n}^a-K_n^a)\one_D\right]
 \le Cnr^2e^{-cn}.
\end{equation}
\end{lemma}

\begin{proof}
Choose \(s<s'<1\). Since \(r=o(1)\), the disk \(D_r(a)\) is
contained in \(D_{s'}\) for all large \(n\), uniformly over the
stated range.  The projection remainder
\[
 E_n=K_{\infty,n}-K_n
 =\sum_{j=n}^{\infty}\varphi_{j,n}\otimes\varphi_{j,n}
\]
is positive.  Its diagonal is
\[
 E_n(w,w)=\frac n\pi e^{-n|w|^2}
 \sum_{j=n}^{\infty}\frac{(n|w|^2)^j}{j!}
 =\frac n\pi\Pp\{\Pois(n|w|^2)\ge n\}.
\]
For \(|w|\le s'\), the exponential Markov inequality, optimized at
\(\theta=\log(1/|w|^2)\), yields
\[
 \Pp\{\Pois(n|w|^2)\ge n\}
 \le\exp\{-n[\log(s'^{-2})-1+s'^2]\}.
\]
The bracket is strictly positive.  Positivity of \(E_n\) and the
two-by-two principal minor inequality give
\begin{equation}\label{eq:EnCSexpanded}
 |E_n(z,w)|^2\le E_n(z,z)E_n(w,w),
\end{equation}
hence the same exponential estimate holds off the diagonal in the bulk.

It remains to compare the rank-one terms in the Palm formula.  Set
\[
 \alpha=K_{\infty,n}(a,a),\quad\beta=K_n(a,a),\quad
 u=K_{\infty,n}(w,a),\quad v=K_n(w,a).
\]
Here \(\alpha=n/\pi\), \(\beta\ge n/(2\pi)\) for large \(n\),
\(|u|,|v|\le n/\pi\), and \eqref{eq:EnCSexpanded} gives
\(|u-v|+|\alpha-\beta|\le Cne^{-cn}\), after changing \(c\). Therefore
\begin{align*}
 \left|\frac{|u|^2}{\alpha}-\frac{|v|^2}{\beta}\right|
 &\le\frac{(|u|+|v|)|u-v|}{\alpha}
 +\frac{|v|^2|\alpha-\beta|}{\alpha\beta}
 \le Cne^{-cn}.
\end{align*}
Subtracting the two Palm kernels proves the upper bound in
\eqref{eq:palmtaildiagonal}; the lower bound follows from the projection
inclusion \(\mathcal H_n^a\subset\mathcal H_{\infty,n}^a\). Equation
\eqref{eq:untaildiagonal} is the unconditioned diagonal estimate at \(a\).
Finally, a positive continuous-kernel operator has trace equal to the
integral of its diagonal.  Integration over \(D_r(a)\) proves
\eqref{eq:palmtailtrace}.
\end{proof}

\subsection{Relative comparison of the hole probabilities}

\begin{lemma}[Relative Palm determinant comparison]\label{lem:palmcompare}
Fix \(s<1\) and \(C_0<\infty\). There are \(c,C>0\) such that, uniformly for
\(|a|\le s\) and \(0\le r\le C_0n^{-1/2}(\log n)^{1/4}\),
\begin{equation}\label{eq:palmcompare}
 0\le\log\frac{q_n(a,r)}{q(\sqrt n\,r)}\le Ce^{-cn+nr^2}.
\end{equation}
In particular, \(q_n(a,r)/q(\sqrt n\,r)=1+O(e^{-c'n})\) uniformly in this range.
Also
\begin{equation}\label{eq:densitycompare}
 K_n(a,a)=\frac n\pi\bigl(1+O(e^{-c'n})\bigr).
\end{equation}
\end{lemma}

\begin{proof}
All operators in this proof are in the unit-disk, variance-\(1/n\)
coordinates of \eqref{eq:Kn}. Since
\(\mathcal H_n^a\subset\mathcal H_{\infty,n}^a\), their orthogonal projections
satisfy \(K_n^a\le K_{\infty,n}^a\). For \(D=D_r(a)\), set
\[
 A=\one_DK_n^a\one_D,\qquad C=\one_DK_{\infty,n}^a\one_D,
 \qquad \mathcal D=C-A.
\]
Then \(0\le A\le C\), \(\mathcal D\ge0\). The strict contraction property
\(\|C\|<1\) follows from the explicit spectrum computed below. Consequently,
\begin{equation}\label{eq:detorder}
 q_n(a,r)=\det(I-A)\ge\det(I-C)=q(\sqrt n\,r).
\end{equation}

By \Cref{lem:palmtailkernel},
\begin{equation}\label{eq:Dtrace}
 \|\mathcal D\|_1=\Tr\mathcal D\le Cnr^2e^{-cn}.
\end{equation}

Magnetic translation is a unitary equivalence between \(C\) and the
restriction of the microscopic reduced-Palm kernel at \(0\) to
\(D_{\sqrt n r}(0)\). By the proof
of \Cref{lem:palmproduct}, its eigenvalues are
\(\lambda_j=\gamma(j+1,nr^2)/j!\), \(j\ge1\). The gamma--Poisson identity
gives
\[
 \lambda_j=\Pp\{\Pois(nr^2)\ge j+1\},
\]
so they decrease with \(j\), and all are strictly below one.
Hence
\begin{equation}\label{eq:resolvent}
 \|(I-C)^{-1}\|=\frac1{1-\lambda_1}
 =\frac{e^{nr^2}}{1+nr^2}.
\end{equation}
Apply \Cref{lem:abstractdetcomparison}, \eqref{eq:Dtrace}, and
\eqref{eq:resolvent} to obtain
\[
 0\le\log\frac{q_n(a,r)}{q(\sqrt n\,r)}
 \le\frac{e^{nr^2}}{1+nr^2}\,Cnr^2e^{-cn}
 \le Ce^{-cn+nr^2}.
\]
This proves \eqref{eq:palmcompare}. Equation \eqref{eq:densitycompare} is
\eqref{eq:untaildiagonal}.
\end{proof}

\section{Proofs of the extreme-value theorems}

To shorten the proofs only, set
\[
 L_n:=\log\frac{n|B|}{\pi}.
\]

\subsection{From adaptive radii to a common threshold}

The main probabilistic input concerns the spatially varying radius
\(R_{n,\kappa}(a)\), whereas the maximum in \eqref{eq:Mn} is evaluated at a
common radius.  The following deterministic observation isolates the
sandwich argument used to pass between the two.

For \(t\ge0\), write
\begin{equation}\label{eq:exceedancecount}
 N_n(t;B)=\#\{a\in\xi_n\cap B:d_n(a)>t\}.
\end{equation}

\begin{lemma}[Uniform threshold transfer]\label{lem:thresholdtransfer}
Let \(t_n\ge0\), \(\kappa>0\), and suppose that
\begin{equation}\label{eq:thresholdtransferassumption}
 \sup_{a\in B}
 \left|K_n(a,a)q_n(a,t_n)-\kappa\right|\longrightarrow0.
\end{equation}
Then
\begin{equation}\label{eq:thresholdtransfervoid}
 \Pp\{N_n(t_n;B)=0\}\longrightarrow e^{-\kappa|B|},
\end{equation}
and, for every fixed \(k\ge1\),
\begin{equation}\label{eq:thresholdtransferorder}
 \Pp\{N_n(t_n;B)\le k-1\}
 \longrightarrow e^{-\kappa|B|}
 \sum_{j=0}^{k-1}\frac{(\kappa|B|)^j}{j!}.
\end{equation}
\end{lemma}

\begin{proof}
Fix \(\varepsilon\in(0,1)\).  By
\eqref{eq:thresholdtransferassumption}, for all sufficiently large \(n\),
uniformly in \(a\in B\),
\[
 \kappa(1-\varepsilon)<K_n(a,a)q_n(a,t_n)
 <\kappa(1+\varepsilon).
\]
Because \(r\mapsto q_n(a,r)\) is strictly decreasing,
\[
 R_{n,\kappa(1+\varepsilon)}(a)
 \le t_n\le R_{n,\kappa(1-\varepsilon)}(a).
\]
An anchor retained at the larger radius is also retained at the smaller
radius.  Hence, with the notation of \Cref{prop:LO},
\[
 N_{n,\kappa(1+\varepsilon)}(B)
 \ge N_n(t_n;B)
 \ge N_{n,\kappa(1-\varepsilon)}(B).
\]
Apply \Cref{prop:LO} to the outside counts.  Their limiting means are
\(\kappa(1+\varepsilon)|B|\) and
\(\kappa(1-\varepsilon)|B|\).  The Poisson probabilities of the events
\(\{0\}\) and \(\{0,\ldots,k-1\}\) are continuous functions of the mean.
Letting \(\varepsilon\downarrow0\) gives
\eqref{eq:thresholdtransfervoid} and
\eqref{eq:thresholdtransferorder}.
\end{proof}

We shall also use the elementary moving-argument consequence of convergence
of distribution functions.  If \(F_n(y)\to F(y)\) for every \(y\), and
\(F\) is continuous, then for every deterministic \(y_n\to y\),
\begin{equation}\label{eq:movingargument}
 F_n(y_n)\longrightarrow F(y).
\end{equation}
Indeed, for each \(\delta>0\), monotonicity gives
\(F_n(y-\delta)\le F_n(y_n)\le F_n(y+\delta)\) for all large \(n\);
first let \(n\to\infty\), then \(\delta\downarrow0\).  This small point is
needed below when the deterministic inversion produces \(x+o(1)\) rather
than exactly \(x\).

\begin{proof}[Proof of \Cref{thm:main}]
Fix \(x\in\R\), and choose \(s<1\) with \(\overline B\subset D_s\). For
all sufficiently large \(n\), let
\begin{equation}\label{eq:tn}
 t_n(x)=\frac1{\sqrt n}\Psi^{-1}(L_n+x),\qquad
 \kappa_x=\frac{e^{-x}}{|B|}.
\end{equation}
By definition,
\begin{equation}\label{eq:qtn}
 q(\sqrt n\,t_n(x))=\frac{\pi e^{-x}}{n|B|}.
\end{equation}
The leading term in \eqref{eq:palmexpansion} gives
\(n^2t_n(x)^4=4L_n(1+o(1))\), and in particular
\(t_n(x)=O(n^{-1/2}(\log n)^{1/4})\). Therefore
\Cref{lem:palmcompare}, \eqref{eq:densitycompare}, and \eqref{eq:qtn} imply
\begin{equation}\label{eq:thresholduniform}
 \sup_{a\in B}
 \left|K_n(a,a)q_n(a,t_n(x))-\kappa_x\right|=o(1).
\end{equation}

Apply \Cref{lem:thresholdtransfer} with \(t_n=t_n(x)\) and
\(\kappa=\kappa_x\).  Since \(\kappa_x|B|=e^{-x}\) and
\(\{M_n(B)\le t_n(x)\}=\{N_n(t_n(x);B)=0\}\), this proves
\eqref{eq:mainlimit}. Open versus closed
balls and strict versus weak inequalities do not affect the result because
the joint eigenvalue law has a density.
\end{proof}

\subsection{Deterministic inversion of the hole exponent}

\begin{lemma}[Deterministic inversion through order one]
\label{lem:inversion}
Let \(L\to\infty\), put \(u=2\sqrt L\), \(\ell=\log u\), and define
\begin{equation}\label{eq:betacompact}
 \beta=\beta(L):=u^2-2u\ell+4a_2u+4a_1u^{1/2}
 +2\ell^2+\left(\frac{10}{3}-8a_2\right)\ell
 +8a_2^2-4a_2+4a_0.
\end{equation}
Uniformly for \(x\) in a fixed compact subset of \(\R\),
\begin{equation}\label{eq:inversionresidual}
 \Phi\bigl((\beta+4x)^{1/4}\bigr)
 =L+x+O\!\left(\frac{\ell^3}{\sqrt u}\right).
\end{equation}
\end{lemma}

\begin{proof}
Set \(s_x=\beta+4x\), \(y_x=\sqrt{s_x}\), and abbreviate
\(c=a_2\), \(a=a_1\), \(d=a_0\). Expanding
\(y_x=u\sqrt{1+(s_x-u^2)/u^2}\) through the terms that can contribute at
order one gives
\begin{align}
 y_x={}&u-\ell+2c+2a u^{-1/2}\notag\\
 &+\frac1u\left\{\frac12\ell^2+\left(\frac53-2c\right)\ell
        +2c^2-2c+2d+2x\right\}
 +O\!\left(\ell u^{-3/2}+\ell^3u^{-2}\right).
 \label{eq:yxexpansion}
\end{align}
We give the coefficient bookkeeping in detail.  Write
\[
 s_x=u^2+v_x,
\]
where, directly from \eqref{eq:betacompact},
\begin{align*}
 v_x={}&-2u\ell+4cu+4au^{1/2}+2\ell^2
 +\left(\frac{10}{3}-8c\right)\ell\\
 &+8c^2-4c+4d+4x.
\end{align*}
Then \(v_x=O(u\ell)\), uniformly for bounded \(x\).  The Taylor formula
\[
 \sqrt{u^2+v}=u+\frac{v}{2u}-\frac{v^2}{8u^3}
 +O\!\left(\frac{|v|^3}{u^5}\right),
 \qquad |v|\le\frac12u^2,
\]
shows first that
\[
 \frac{v_x}{2u}=-\ell+2c+2au^{-1/2}
 +\frac1u\left\{\ell^2+\left(\frac53-4c\right)\ell
 +4c^2-2c+2d+2x\right\}.
\]
Only the leading part \(-2u\ell+4cu\) of \(v_x\) contributes to
\(v_x^2/(8u^3)\) at order \(u^{-1}\).  The square term therefore contributes
\(-\tfrac12(\ell-2c)^2/u\); combining it with the constant line of
\eqref{eq:betacompact} yields the braces in \eqref{eq:yxexpansion}.

Write \(h=y_x-u\) and denote the expression in braces by \(e_x\). Since
\(L=u^2/4\), formula \eqref{eq:Phi} becomes
\begin{align}
 \Phi(\sqrt{y_x})-L
={}&\frac{uh}{2}+\frac{h^2}{4}
 +\frac{u+h}{2}\log(u+h)-c(u+h)\notag\\
 &-a\sqrt{u+h}-\frac13\log(u+h)-d.\label{eq:inversionsub}
\end{align}
Here \(h=O(\ell)\), so, uniformly for bounded \(x\),
\[
 \log(u+h)=\ell+\frac hu+O\!\left(\frac{\ell^2}{u^2}\right),
 \qquad
 \sqrt{u+h}=\sqrt u+\frac{h}{2\sqrt u}
 +O\!\left(\frac{\ell^2}{u^{3/2}}\right).
\]
In \eqref{eq:inversionsub}, the terms of order \(u\ell\) and \(u\) cancel
because the first correction in \(h\) is \(-\ell+2c\); the terms of order
\(\sqrt u\) cancel because the next correction is \(2a u^{-1/2}\). The
remaining order-one expression is
\[
 \frac{e_x}{2}-\frac{\ell^2}{4}
 +(c-\tfrac56)\ell-c^2+c-d=x.
\]
For reference, the cancellation occurs successively at the following
orders.  The correction \(-\ell\) removes the \(u\ell\) contribution,
\(2c\) removes the remaining order \(u\) term, and
\(2au^{-1/2}\) removes the order \(\sqrt u\) term.  The five terms in
\(e_x\) then cancel, respectively, the coefficients of \(\ell^2\),
\(\ell\), the constant depending on \(c\), the constant \(d\), and leave
exactly \(x\).  The remainders in the Taylor formulae and in
\eqref{eq:yxexpansion} are bounded by the error in
\eqref{eq:inversionresidual}.  Since \(\ell^3/\sqrt u\to0\), this error is
indeed \(o(1)\).
\end{proof}

\begin{proof}[Proof of \Cref{thm:explicit}]
By \Cref{thm:palmexpansion},
\begin{equation}\label{eq:PsiPhi}
 \Psi(r)=\Phi(r)+O(r^{-1/6}).
\end{equation}
The inlined expression in \eqref{eq:beta} is exactly
\(\beta(L_n)\) from \eqref{eq:betacompact}.
Apply \Cref{lem:inversion} with \(L=L_n\). Since
\((\beta_n+4x)^{1/4}\asymp L_n^{1/4}\), \eqref{eq:PsiPhi} and
\eqref{eq:inversionresidual} give
\begin{equation}\label{eq:betacheck}
 \Psi\bigl((\beta_n+4x)^{1/4}\bigr)=L_n+x+o(1).
\end{equation}
Apply \eqref{eq:movingargument} to the distribution functions in
\Cref{thm:main} at the physical threshold
\(n^{-1/2}(\beta_n+4x)^{1/4}\) and use \eqref{eq:betacheck}; this proves
\eqref{eq:fourth} without assuming uniformity in \(x\) in
\Cref{thm:main}.

To obtain the affine form, put
\(U_n=\sqrt n\,M_n(B)\) and
\(Y_n=(U_n^4-\beta_n)/4\). The fourth-power limit gives
\(Y_n\Longrightarrow\Gum\), hence \(Y_n=O_{\Pp}(1)\) and
\(U_n/\beta_n^{1/4}\to1\) in probability. Factoring the difference of
fourth powers therefore yields
\begin{align*}
 &\beta_n^{3/4}\bigl(U_n-\beta_n^{1/4}\bigr)\\
 &\quad=
 \frac{\beta_n^{3/4}\bigl(U_n^4-\beta_n\bigr)}
 {U_n^3+U_n^2\beta_n^{1/4}
  +U_n\beta_n^{1/2}+\beta_n^{3/4}}
 =Y_n+o_{\Pp}(1),
\end{align*}
which proves \eqref{eq:affine}.

Finally, \eqref{eq:yxexpansion} at \(x=0\) gives directly
\[
 \sqrt{\beta_n}
 =2\sqrt{L_n}-\log(2\sqrt{L_n})+2a_2+o(1)
 =2\sqrt{L_n}-\frac12\log L_n+4-\log(4\pi)+o(1),
\]
which is the deterministic expansion preceding
\eqref{eq:simplecentering}. Moreover,
\[
 nM_n(B)^2-\sqrt{\beta_n}
 =\frac{n^2M_n(B)^4-\beta_n}{nM_n(B)^2+\sqrt{\beta_n}}
 =o_{\Pp}(1),
\]
and combining the last two displays proves \eqref{eq:simplecentering}.
\end{proof}

\begin{proof}[Proof of \Cref{cor:kth}]
The event \(\{\Psi(\sqrt n\,M_{n,k})-L_n\le x\}\) says that at most \(k-1\) anchors
exceed the radius \(t_n(x)\), namely
\(N_n(t_n(x);B)\le k-1\). The count inequalities above give
\begin{align*}
 \{N_{n,\kappa_x(1+\varepsilon)}(B)\le k-1\}
 &\subset\{N_n(t_n(x);B)\le k-1\}\\
 &\subset\{N_{n,\kappa_x(1-\varepsilon)}(B)\le k-1\}.
\end{align*}
By \Cref{prop:LO}, the outside counts converge to Poisson variables with
means \(e^{-x}(1+\varepsilon)\) and \(e^{-x}(1-\varepsilon)\). Their
distribution functions at the integer \(k-1\) are continuous in the mean.
Letting \(\varepsilon\downarrow0\) gives the claimed formula.
\end{proof}

\appendix
\section{Compact uniformity in the disk-hole expansion}
\label{app:uniform}

We prove \Cref{lem:uniformdisk}. Throughout this appendix,
\[
 I=[t_-,t_+]\Subset(0,1),\qquad x=Nt^2,\qquad t\in I.
\]
Choose once and for all \(\epsilon\in(0,1/2)\) so that
\(t_+^2/(1-\epsilon)<1\), and set
\[
 j_-:=\left\lceil\frac{x}{1+\epsilon}\right\rceil,
 \qquad j_+:=\left\lfloor\frac{x}{1-\epsilon}\right\rfloor,
 \qquad
 \vartheta:=j_--\frac{x}{1+\epsilon}\in[0,1).
\]
For all sufficiently large \(N\), uniformly in \(t\in I\),
\(1\le j_-\le j_+\le N\). We use the exact decomposition
\begin{equation}\label{eq:threeblocks}
 \sum_{k=1}^N\log Q(k,x)=S_<+S_0+S_>,
\end{equation}
where
\[
 S_<:=\sum_{k=1}^{j_--1}\log Q(k,x),\quad
 S_0:=\sum_{k=j_-}^{j_+}\log Q(k,x),\quad
 S_>:=\sum_{k=j_++1}^{N}\log Q(k,x).
\]
In the proof of \cite[Theorem 1.7]{CharlierAnnuli}, specialised to
\(g=1,b=1,\alpha=0\), these are respectively the blocks \(S_3,S_4,S_5\).
Keeping \(\vartheta\) is essential: its terms cancel only after \(S_<\) and
\(S_0\) have been added.

We begin by isolating the elementary asymptotic input.  This also explains
why the expansions below remain uniform when the disk parameter varies in a
compact subset of the bulk.

\begin{lemma}[Uniform Barnes, Stirling, and trigamma expansions]
\label{lem:uniformelementary}
Let \(0<c<C<\infty\).  Uniformly for \(y\in[cN,CN]\),
\begin{align}
 \log G(y+1)
 &=\frac{y^2}{2}\log y-\frac{3y^2}{4}
 +\frac y2\log(2\pi)-\frac1{12}\log y
 +\zeta'(-1)+O_{c,C}(N^{-2}),\label{eq:uniformBarnes}\\
 \log\Gamma(y)
 &=\left(y-\frac12\right)\log y-y+\frac12\log(2\pi)
 +\frac1{12y}+O_{c,C}(N^{-3}),\label{eq:uniformStirling}\\
 \psi_1(y)
 &=\frac1y+\frac1{2y^2}+\frac1{6y^3}
 +O_{c,C}(N^{-5}).\label{eq:uniformtrigamma}
\end{align}
The bounds are unchanged if \(y\) is perturbed by a quantity in a fixed
bounded interval.
\end{lemma}

\begin{proof}
These are the standard large-positive-argument expansions of the Barnes
\(G\)-function, the gamma function, and the trigamma function, with their
first omitted terms used as remainder bounds.  On the ray \(y\ge cN\), the
usual remainders are bounded respectively by constant multiples of
\(y^{-2}\), \(y^{-3}\), and \(y^{-5}\).  Replacing powers of \(y^{-1}\)
by powers of \(N^{-1}\) gives the displayed estimates.  If \(|h|\le C_0\),
then \(y+h\in[cN/2,2CN]\) for all sufficiently large \(N\), uniformly in
\(h\), which proves the final assertion.
\end{proof}

\begin{lemma}[The two outer blocks]\label{lem:outerblocks}
Uniformly for \(t\in I\),
\begin{equation}\label{eq:Sgreater}
 S_>=O_I(e^{-c_I N}).
\end{equation}
Moreover,
\begin{equation}\label{eq:Sless}
 S_<=F_1N^2+F_2N\log N+F_3N+F_5\log N+F_6
       +O_I\!\left(\frac{\log N}{N}\right),
\end{equation}
where
\begin{align*}
 F_1={}&-\frac{t^4}{4(1+\epsilon)^2}
       \bigl(1+4\epsilon-2\log(1+\epsilon)\bigr), \qquad  F_2=-\frac{t^2}{2(1+\epsilon)},\\
 F_3={}&\frac{t^2}{1+\epsilon}\left[
 (1-\vartheta)\epsilon+1-\log(t\sqrt{2\pi})
 +\epsilon\log\frac{\epsilon}{1+\epsilon}
 +(\vartheta-1)\log(1+\epsilon)\right],\\
 F_5={}&\frac7{12}-\frac{\vartheta}{2},\\
 F_6={}&\frac{-1+6\vartheta^2}{12}\log(1+\epsilon)-\frac1\epsilon
 +\frac{1-2\vartheta}{2}\log\epsilon+\frac76\log t
 +\frac12\log(2\pi)\\
 &\hspace{17mm}-\vartheta\log(t\sqrt{2\pi})-\zeta'(-1).
\end{align*}
\end{lemma}

\begin{proof}
For \(k\ge j_++1\), the gamma--Poisson identity and Chernoff's bound imply
\[
 0\le1-Q(k,x)=\Pp\{\Pois(x)\ge k\}\le e^{-c_\epsilon x}.
\]
Since \(x\ge t_-^2N\), and \(-\log(1-u)\le2u\) for \(0\le u\le1/2\),
summing at most \(N\) terms proves \eqref{eq:Sgreater}.

Put \(m=j_--1=x/(1+\epsilon)+\vartheta-1\). Since \(k\le m\) implies
\(x-k\ge c_\epsilon x\), the upper incomplete-gamma expansion
\cite[Lemma 5.1]{CharlierAnnuli} is uniform in this range and gives
\[
 Q(k,x)=\frac{x^ke^{-x}}{\Gamma(k)}
 \left(\frac1{x-k}-\frac{x}{(x-k)^3}+R_{k,x}\right),
 \qquad |R_{k,x}|\le C_\epsilon x^{-3}.
\]
Taking the logarithm therefore yields, uniformly for \(1\le k\le m\),
\begin{equation}\label{eq:logQouter}
 \log Q(k,x)=k\log x-x-\log\Gamma(k)-\log(x-k)
 -\frac{x}{(x-k)^2}+O_\epsilon(x^{-2}).
\end{equation}
Summing and using
\[
 \sum_{k=1}^m\log\Gamma(k)=\log G(m+1),\qquad
 \sum_{k=1}^m\log(x-k)=\log\Gamma(x)-\log\Gamma(x-m),
\]
\[
 \sum_{k=1}^m\frac1{(x-k)^2}=\psi_1(x-m)-\psi_1(x)
\]
gives
\begin{align}\label{eq:Slessclosed}
 S_<={}&\frac{m(m+1)}2\log x-mx-\log G(m+1)
 -\log\Gamma(x)+\log\Gamma(x-m)\notag\\
 &-x\{\psi_1(x-m)-\psi_1(x)\}+O_I(N^{-1}).
\end{align}
Here \(G\) is Barnes' \(G\)-function and \(\psi_1\) is the trigamma
function. Their arguments are bounded above and below by positive multiples
of \(N\), uniformly for \(t\in I\), so
\Cref{lem:uniformelementary} applies.  To make the moving-endpoint
bookkeeping explicit, put
\[
 \alpha=(1+\epsilon)^{-1},\qquad
 \delta=\vartheta-1\in[-1,0),\qquad m=\alpha x+\delta.
\]
Uniformly in \(\vartheta\in[0,1]\), Taylor expansion gives
\begin{align}
 \log m
 &=\log x-\log(1+\epsilon)
 +\frac{(1+\epsilon)\delta}{x}
 -\frac{(1+\epsilon)^2\delta^2}{2x^2}
 +O_\epsilon(x^{-3}),\label{eq:logmendpoint}\\
 \log(x-m)
 &=\log x+\log\frac{\epsilon}{1+\epsilon}
 -\frac{(1+\epsilon)\delta}{\epsilon x}
 -\frac{(1+\epsilon)^2\delta^2}{2\epsilon^2x^2}
 +O_\epsilon(x^{-3}).\label{eq:logxmendpoint}
\end{align}
Insert \eqref{eq:uniformBarnes}--\eqref{eq:uniformtrigamma} and
\eqref{eq:logmendpoint}--\eqref{eq:logxmendpoint} into
\eqref{eq:Slessclosed}.  Because \(x=Nt^2\) and \(t\in I\), every omitted
term is uniform in \(t\) and \(\vartheta\).  Collecting the coefficients of
\(N^2\), \(N\log N\), \(N\), \(\log N\), and \(1\) gives precisely
\(F_1,F_2,F_3,F_5,F_6\).  This is also the specialisation of
\cite[Lemma 5.2]{CharlierAnnuli}.
\end{proof}

\begin{lemma}[The transition block]\label{lem:transitionblock}
Let
\begin{align*}
 \mathcal A={}&\int_{-\infty}^0\log\!\left(\frac{\erfc y}{2}\right)\dd y+\int_0^\infty\left[\log\!\left(\frac{\erfc y}{2}\right)
 +y^2+\log y+\log(2\sqrt\pi)\right]\dd y.
\end{align*}
Uniformly for \(t\in I\),
\begin{equation}\label{eq:Szero}
 S_0=E_1N^2+E_2N\log N+E_3N+E_4\sqrt N+E_5\log N+E_6
       +O_I(N^{-1/12}),
\end{equation}
where
\begin{align*}
 E_1={}&\frac{t^4}{4(1+\epsilon)^2}
       \bigl(2\epsilon-\epsilon^2-2\log(1+\epsilon)\bigr), \qquad E_2=-\frac{t^2\epsilon}{2(1+\epsilon)},\\
 E_3={}&\frac{t^2}{1+\epsilon}\left[
 (1+\epsilon-\vartheta)\log(1+\epsilon)
 +\epsilon\vartheta-\epsilon\log(\epsilon t\sqrt{2\pi})\right], \qquad
 E_4=\sqrt2\,t\mathcal A=a_1t,\\
 E_5={}&\frac{2\vartheta-1}{4}, \qquad  E_6=\frac{1-6\vartheta^2}{12}\log(1+\epsilon)+\frac1\epsilon
 +\frac{2\vartheta-1}{2}\log(\epsilon t\sqrt{2\pi})-2J_--2J_+.
\end{align*}
\end{lemma}

\begin{proof}
This is the \(k=2,b=1,\alpha=0,r_k=t\) specialisation of the final
transition-block formula in \cite[Lemma 3.26]{CharlierAnnuli}. We verify the
only point not explicit in that fixed-parameter statement: uniformity of its
remainder for \(t\in I\).

Set \(M=N^{1/12}\) and introduce
\[
 g_-:=\left\lceil\frac{x}{1+M/\sqrt N}\right\rceil,
 \qquad
 g_+:=\left\lfloor\frac{x}{1-M/\sqrt N}\right\rfloor.
\]
The transition block is split exactly into the sums over
\([j_-,g_--1]\), \([g_-,g_+]\), and \([g_++1,j_+]\), as in
\cite[(3.38)--(3.39)]{CharlierAnnuli}. In the middle sum put
\(v_k=\sqrt N(x/k-1)\). The uniform incomplete-gamma expansion in
\cite[Lemmas 2.3--2.4]{CharlierAnnuli}, differentiated there by analyticity
and Cauchy's formula, gives the central summand expansions used in
\cite[Lemmas 3.11 and 3.14]{CharlierAnnuli}; in particular,
\cite[(3.46)--(3.52)]{CharlierAnnuli} records the relevant terms. Their
pointwise remainder is bounded by
\[
 C_I(1+|v_k|^8)N^{-3/2},\qquad |v_k|\le M.
\]
The mesh and the number of summands can be tracked explicitly. Throughout
the central window,
\[
 \frac{k}{N}\in
 \left[
 \frac{t_-^2}{1+\epsilon}+O(N^{-1}),
 \frac{t_+^2}{1-\epsilon}+O(N^{-1})
 \right]\Subset(0,1),
\]
and
\[
 |v_{k+1}-v_k|
 =\frac{\sqrt N\,x}{k(k+1)}\asymp_I N^{-1/2}.
\]
There are therefore \(O_I(M\sqrt N)\) central summands. Summing the
pointwise polynomial remainder gives
\[
 O_I\!\left(N^{-3/2}\sum_{|v_k|\le M}(1+|v_k|^8)\right)
 =O_I(M^9/N)=O_I(N^{-1/4}).
\]

For the two matching sums, the estimates assembled in
\cite[Lemmas 3.16, 3.19--3.20 and 3.24--3.26]{CharlierAnnuli}, after
subtraction of the
polynomial and logarithmic tails of \(\log(\erfc y/2)\), give
\begin{equation}\label{eq:matchingerror}
 O_t\!\left(\frac{M^5}{\sqrt N}\right)
 +O_t\!\left(\frac{\sqrt N}{M^7}\right).
\end{equation}
The constants are uniform on \(I\): throughout those proofs, \(k/N\) stays
in the fixed compact interval
\([t_-^2/(1+\epsilon),t_+^2/(1-\epsilon)]\Subset(0,1)\); all fractional endpoint
parameters lie in \([0,1]\); and all coefficients and all finitely many
derivative majorants are continuous functions of \(t,t^{-1}\), and their
logarithms on \(I\). The complementary-error-function ratios are uniformly
controlled after the rescaling \(y=tv/\sqrt2\). More explicitly,
denominators involving
\(k/N\) and \(1-k/N\) are bounded below by
\(t_-^2/(1+\epsilon)\) and
\(1-t_+^2/(1-\epsilon)\), respectively. Denominators involving
\(x/k-1\) are absorbed into the transition coordinate \(v\) and are
controlled by the displayed polynomial majorants.  The rescaling
\(y=tv/\sqrt2\) is uniformly bi-Lipschitz on \(I\).  The standard
Mills-ratio estimates
\[
 e^{y^2}\erfc(y)\asymp(1+y)^{-1},\qquad y\ge0,
\]
together with their differentiated forms, provide uniform polynomial
majorants for the complementary-error-function quotients occurring in the
Euler--Maclaurin remainders. Thus every \(O_t\)-constant in the cited finite
list has bounded supremum on \(I\).

For clarity, the complete transition error budget is
\[
\begin{array}{c|c|c}
 \text{source}&\text{bound for general }M&\text{bound for }M=N^{1/12}\\ \hline
 \text{central summands}&M^9/N&N^{-1/4}\\
 \text{first matching range}&M^5/\sqrt N&N^{-1/12}\\
 \text{second matching range}&\sqrt N/M^7&N^{-1/12}\\
 \text{renormalized tail}&M^{-2}&N^{-1/6}\\
 \text{large-deviation tail}&Ne^{-cM^2}&Ne^{-cN^{1/6}}.
\end{array}
\]
Thus \(N^{-1/12}\) is the dominant error. Specialising the coefficients in
\cite[Lemma 3.26]{CharlierAnnuli} yields
the displayed \(E_1,\ldots,E_6\).
\end{proof}

\begin{proof}[Completion of the proof of \Cref{lem:uniformdisk}]
Add the three blocks in \eqref{eq:threeblocks}. Since \(S_>\) is
exponentially small, \Cref{lem:outerblocks,lem:transitionblock} give
\[
 \sum_{k=1}^{N}\log Q(k,Nt^2)
 =\sum_{\nu\in\{1,2,3,5,6\}}(F_\nu+E_\nu)\,\Xi_\nu
   +E_4\sqrt N+O_I(N^{-1/12}),
\]
where \(\Xi_1=N^2\), \(\Xi_2=N\log N\), \(\Xi_3=N\),
\(\Xi_5=\log N\), and \(\Xi_6=1\). Direct algebra gives
\begin{align*}
 F_1+E_1&=-\frac{t^4}{4},
 &F_2+E_2&=-\frac{t^2}{2},\\
 F_3+E_3&=t^2\bigl(1-\log(t\sqrt{2\pi})\bigr),
 &E_4&=a_1t,\\
 F_5+E_5&=\frac13,
 &F_6+E_6&=\frac23\log t+a_0.
\end{align*}
For the constant term, the cancellation can be seen without suppressing any
endpoint contribution. The coefficients of \(\log(1+\epsilon)\),
\(1/\epsilon\), and \(\log\epsilon\) cancel separately. The remaining
logarithms equal
\[
 \frac76\log t+\frac12\log(2\pi)
 -\frac12\log(t\sqrt{2\pi})=\frac23\log t+\frac14\log(2\pi),
\]
and the remaining constants are \(-\zeta'(-1)-2J_--2J_+\). Thus every
occurrence of the auxiliary cutoff \(\epsilon\) and of the moving-endpoint
fraction \(\vartheta\) cancels identically. Substituting the six sums above
gives exactly \eqref{eq:uniformdisk}, uniformly for \(t\in I\).

Finally, the integrals defining \(a_1,J_-,J_+\) are absolutely convergent.
The cancellations can be checked directly. As \(y\to+\infty\),
\begin{align*}
 \log\!\left(\frac{\erfc y}{2}\right)
 &=-y^2-\log y-\log(2\sqrt\pi)
   -\frac1{2y^2}+\frac5{8y^4}+O(y^{-6}),\\
 \frac{e^{-y^2}(1-5y^2)}{3\sqrt\pi\,\erfc y}
 &=-\frac53y^3-\frac12y+\frac1y-\frac9{4y^3}
   +O(y^{-5}).
\end{align*}
Consequently, after the counterterms in
\eqref{eq:a1}--\eqref{eq:Jplus} are subtracted, the positive-half-line
integrands are respectively \(-\tfrac12y^{-2}+O(y^{-4})\) and
\(-y^{-3}+O(y^{-5})\). At \(-\infty\),
\[
 \frac{\erfc y}{2}=1+O\!\left(\frac{e^{-y^2}}{|y|}\right),
\]
so the negative-half-line integrands decay exponentially. Near zero, the
only nonsmooth terms are constant multiples of \(\log y\) and
\(y\log y\), both integrable. This proves absolute convergence.
\end{proof}

\section*{Acknowledgements}

P.M. was supported by the Swedish Foundation for International Cooperation
in Research and Higher Education (PD2023-9315).

\section*{Statement on the use of artificial intelligence}

In the course of the research presented here I have been using AI tools
extensively, most significantly ChatGPT Pro 5.5 and ChatGPT 5.6 Sol Ultra,
for ideation, technical help, editing and checking the proofs, and general
editing and proofreading of the manuscript, at the level of a leading co-author. All mistakes are my own responsibility.

\end{document}